\documentclass[11pt]{article}

\usepackage[margin=1in]{geometry}
\usepackage{amsmath,amssymb,amsthm,mathtools}
\usepackage{microtype}
\usepackage[colorlinks=true,linkcolor=blue,citecolor=blue,urlcolor=blue]{hyperref}
\usepackage{xspace}

\newtheorem{theorem}{Theorem}[section]
\newtheorem{lemma}[theorem]{Lemma}

\newtheorem{corollary}[theorem]{Corollary}
\theoremstyle{remark}
\newtheorem{remark}[theorem]{Remark}

\newcommand{\F}{\mathbb{F}}
\newcommand{\Z}{\mathbb{Z}}

\newcommand{\Norm}{\operatorname{N}}

\newcommand{\erdos}{Erd\H{o}s\xspace}

\title{An Improved Upper Bound for Colorings Without Symmetrically Colored $k$-Term Arithmetic Progressions}
\author{
  Ruizhe Shi\\
  University of Washington
  \and
  Yiqi Dong\\
  Tsinghua University
}
\date{\today}

\begin{document}

\maketitle

\begin{abstract}
Given a coloring $c$ and an even $k\ge 4$, a nontrivial $k$-term arithmetic progression~($k$-AP) $a,a+d,\ldots,a+(k-1)d$ is called \emph{symmetrically colored} if $c(a+(i-1)d)=c(a+(k-i)d)$, $\forall i\in[k/2]$. Deng, Tidor, and Zhao asked whether $[N]$ admits a coloring with $N^{o(1)}$ colors and no such 4-APs, and gave an $O(N^{\log_{22}3})$-coloring of $[N]$. We give an $O_k(p)$-coloring of $\Z/p^{k^2/4}\Z$ without such $k$-APs for every even $k\ge 4$ and every prime $p>k$, and hence an $O_k(N^{4/k^2})$-coloring of $[N]$, improving the exponent in the upper bound for $4$-APs from $\log_{22}3$ to $1/4$. The construction combines a carry-control coloring of base-$p$ digits with a layered field norm mapping. Together with Behrend-style product colorings, our result for $4$-APs gives $h(N)\leq N^{1/4+o(1)}$ in \erdos's Problem~160 on coloring every nontrivial 4-AP with at least three colors. This result also yields $\rho_4(\alpha)=O_\varepsilon(\alpha^{5-\varepsilon})$ for every $\varepsilon>0$, improving the bound toward Ruzsa's question. Our result for $k$-APs disproves Gowers' conjectured lower bound for all even $k\ge6$ for the \emph{first} time.
\end{abstract}

\maketitle
\pagestyle{plain}

\section{Introduction}

Given a coloring $c$ and an even $k\ge 4$, we say that a $k$-term arithmetic progression~($k$-AP)
\[
    a,\quad a+d,\quad \ldots,\quad a+(k-1)d
\]
is \emph{nontrivial} if $d\neq0$, and \emph{symmetrically colored} if
\[
    c(a+(i-1)d)=c(a+(k-i)d)\qquad \forall i\in[k/2]=\{1,\ldots,k/2\}\;.
\]
Deng, Tidor, and Zhao~\cite[Conjecture~1.6]{DTZ} proposed the following Ramsey-type problem: study colorings of $[N]$ without nontrivial symmetrically colored 4-APs, and conjectured that $N^{o(1)}$ colors suffice.
Their tensor-power construction, based on a three-coloring of $\mathbb Z/22\mathbb Z$, yields
\[
O\left(N^{\log_{22}3}\right)\;,
\]
where $\log_{22}3\approx 0.355$. To the best of our knowledge, this is the strongest previously known upper bound for this problem. Our main result improves this to $O(N^{1/4})$ via an explicit construction.

\begin{theorem}\label{thm:cyclic}
For every even $k\ge 4$ and every prime $p>k$, there is a $(k-1)^{k^2/4-1}p$-coloring of $\Z/p^{k^2/4}\Z$ with no nontrivial symmetrically colored $k$-APs.
\end{theorem}

\begin{corollary}\label{cor:cyclic-4}
For every prime $p\geq 5$, there is a $27p$-coloring of $\Z/p^4\Z$ with no nontrivial symmetrically colored 4-APs.
\end{corollary}

This has several immediate applications. The first concerns \erdos Problem~160~\cite{Erdos1989}:

\begin{quote}
\emph{Let $h(N)$ be the smallest $k$ such that $[N]$ can be colored with $k$ colors so that every 4-AP contains at least three distinct colors. Estimate $h(N)$.}
\end{quote}

Hunter~\cite{Hunter160} observed that the standard Behrend-style product coloring eliminates all relevant two-color patterns except the symmetric pattern. Combining this observation with our result improves the exponent in the previously known upper bound from $0.355+o(1)$ to $1/4+o(1)$.

\begin{corollary}[An improved upper bound for \erdos~Problem~160]\label{cor:erdos}
One has
\[
h(N)\le N^{1/4+o(1)}\;.
\]
\end{corollary}

Second, let $\rho_4(\alpha)$ denote the minimum asymptotic density of 4-APs in a Fourier-uniform set of density $\alpha$, in the sense of~\cite[Definition~1.3]{DTZ}. For each $k\ge 6$, let $\rho_k(\alpha)$ denote the minimum asymptotic density of $k$-APs in a $U^{k-2}$-uniform set of density $\alpha$, in the sense of~\cite[Definition~1.11]{DTZ}. Ruzsa~\cite[Problem~3.2]{CrootLev} asked whether $\rho_4(\alpha)$ decays faster than every power of $\alpha$. Combining the reduction of \cite{DTZ} with our result improves the exponent in the previously known upper bound from $4.406$ to arbitrarily close to $5$.

\begin{corollary}[An improved upper bound for Ruzsa's question]\label{cor:rho} 
For every $\varepsilon>0$, one has
\[
    \rho_4(\alpha)
    = O_\varepsilon\bigl(\alpha^{5-\varepsilon}\bigr)
    \qquad (0<\alpha<1/2)\;.
\]
\end{corollary}

Gowers~\cite[Conjecture 4.2]{Gowers} conjectured that $\rho_k(\alpha)$ is lower bounded by the random baseline $\alpha^k$ for every even $k\ge 4$. Our result provides, to the best of our knowledge, the first counterexample to this conjecture for every even $k\ge 6$, while Gowers~\cite[Theorem~6]{Gowers2020AUS} had previously disproved it for $k=4$.

\begin{corollary}
\label{cor:gowers}
Gowers' conjectured lower bound~\cite[Conjecture~4.2]{Gowers} is false for every even integer
$k\ge4$. 
\end{corollary}

\section{From $k$-APs to digitwise $k$-APs}

Fix an even $k\ge 4$, and a prime $p>k$. For $x\in\{0,1,\dots,p-1\}$, define
\[
\tau(x):=\left\lfloor\frac{(k-1)x}{p}\right\rfloor\in\{0,1,\ldots,k-2\}\;.
\]
Thus $\tau$ divides $\{0,1,\dots,p-1\}$ into $k-1$ consecutive intervals, each of diameter less than $p/(k-1)$.

\begin{lemma}\label{lem:one-digit}
Let $x_0,x_1,\ldots,x_{k-1}\in\{0,1,\dots,p-1\}$ satisfy
\[
    x_{i+1}\equiv x_i+s\pmod p
    \qquad (i=0,1,\ldots,k-2)
\]
for some $s\in\{0,\dots,p-1\}$. For $i=0,1,\ldots,k-2$, define the carry
\[
    e_i:=\mathbf 1_{\{x_{i+1}<x_i\}}\;.
\]
If $\tau(x_{i-1})=\tau(x_{k-i})$ for every $i\in[k/2]$, then $e_0=e_1=\ldots=e_{k-2}$.
\end{lemma}

\begin{proof}
If $s<(k-2)p/(k-1)$, then, since $\tau(x_{k/2-1})=\tau(x_{k/2})$, necessarily $s<p/(k-1)$. If $s\geq(k-2)p/(k-1)$, we instead consider the reversed order $(x_0',x_1',\ldots,x_{k-1}')=(x_{k-1},x_{k-2},\ldots,x_0)$ with $s'=p-s<p/(k-1)$, and 
\[
    e_i':=\mathbf 1_{\{x_{i+1}'<x_i'\}}=1-\mathbf 1_{\{x_{k-1-i}<x_{k-2-i}\}}\qquad (i=0,1,\ldots,k-2)\;.
\]
Thus it suffices to consider the case $s<p/(k-1)$.

Let $\tau(x_{k/2-1})=\tau(x_{k/2})=i$. Since the step is less than $p/(k-1)$, the point $x_{k/2-2}$ lies either in interval $i$ or $i-1$, while $x_{k/2+1}$ lies either in interval $i$ or $i+1$ (interval indices understood modulo $k-1$). As $\tau(x_{k/2-2})=\tau(x_{k/2+1})$, they must both lie in the same interval $i$. By simple induction we have that all of these points lie in the same interval $i$, and $e_0=e_1=\ldots=e_{k-2}=0$. In the reversed case, the same argument gives $e_i'=0$ for every $i$, and therefore $e_0=e_1=\cdots=e_{k-2}=1$.
\end{proof}

\begin{lemma}[digitwise $k$-AP]\label{lem:carry}
Let $L\geq2$, and let
\[
n_i=a+id\in\Z/p^L\Z
\qquad(i=0,1,\ldots,k-1)\;.
\]
For each $n\in\Z/p^L\Z$, write its standard representative as
\[
n=\sum_{j=0}^{L-1}x_j(n)p^j\;,
\qquad
x_j(n)\in\{0,1,\dots,p-1\}\;,
\]
and set
\[
\mathbf{x}(n):=(x_0(n),\dots,x_{L-1}(n))\in\F_p^L\;.
\]
Suppose that for every $j=0,\dots,L-2$,
\[
\tau(x_j(n_{i-1}))=\tau(x_j(n_{k-i}))
\qquad\forall i\in [k/2]\;.
\]
Then there exists $\mathbf h\in\F_p^L$ such that
\[
\mathbf{x}(n_i)=\mathbf{x}(n_0)+i\mathbf h
\qquad(i=0,1,\ldots,k-1)\;.
\]
Moreover, if $d\neq0$ in $\Z/p^L\Z$, then $\mathbf h\neq0$.
\end{lemma}

\begin{proof}
Iterating Lemma~\ref{lem:one-digit} from digit $0$ to $L-2$ shows inductively that the $k-1$ carries to the next digit are equal at every digit position, and hence the digit vectors $\mathbf x(n_0),\ldots,\mathbf x(n_{k-1})$ form a $k$-AP in $\F_p^L$. And $\mathbf h=0$ would imply $n_0=n_1$ and thus $d=0$.
\end{proof}

\section{The field norm construction}

Choose an irreducible polynomial $f(T)\in\F_p[T]$ of odd degree $t\ge 3$, and let $\alpha$ be one of its roots in $\F_{p^{t}}$. Then $\{1,\alpha,\ldots,\alpha^{t-1}\}$ is an $\F_p$-basis of $\F_{p^{t}}$. Identifying $(x_0,x_1,\ldots,x_{t-1})\in\F_p^{t}$ with
\[
x_0+x_1\alpha+\ldots+x_{t-1}\alpha^{t-1}\;,
\]
define
\[
P_t(x_0,x_1,\ldots,x_{t-1})
:=
\Norm_{\F_{p^{t}}/\F_p}
\bigl(x_0+x_1\alpha+\ldots+x_{t-1}\alpha^{t-1}\bigr)\;,
\]
where
\[
\Norm_{\F_{p^{t}}/\F_p}(z):=z^{1+p+\ldots+p^{t-1}}.
\]
The existence of $f(T)$, as well as the fact that $P_t:\F_p^{t}\rightarrow\F_p$ is a \emph{homogeneous} polynomial of \emph{degree $t$} with \emph{no nontrivial zero}, is standard in the theory of finite fields; see~\cite[Corollary~2.11; Theorem~2.28]{Lidl_Niederreiter_1996}. We also let $P_1: \F_p\rightarrow\F_p$ be the identity map. The cubic norm construction $P_3$ is also used in a concurrent result~\cite{Itabe2026OneThird} to claim an intermediate $O(N^{1/3})$ upper bound.

\section{The layered norm construction}
\label{sec:quarter}
For $\mathbf x=(x_0,x_1,\ldots,x_{k^2/4-1})\in\F_p^{k^2/4}$, we decompose
\[
    \mathbf x=\left(\mathbf u^{(1)},\mathbf u^{(3)},\ldots,\mathbf u^{(k-1)}\right)\;,
\]
where $\mathbf u^{(i)}:=\left(x_{(i-1)^2/4},x_{(i-1)^2/4+1},\ldots, x_{(i+1)^2/4-1}\right)\in\F_p^i$. Then we define
\[
    Q(\mathbf x):=\sum_{1\le i\le k-1,\;i\text{ odd}}P_{i}(\mathbf u^{(i)})\;.
\]
A closely related polynomial construction appears in \cite{Fuhrer2024LargeLS} for producing large line-evasive sets, where norm layers of both odd and even degrees are used.
\begin{lemma}[layered norm]\label{lem:layered-norm}
Let
\[
    \mathbf y_i=\mathbf y_0+i\mathbf h
    \qquad(i=0,1,\ldots,k-1)
\]
be a $k$-AP in $\F_p^{k^2/4}$. If $Q(\mathbf{y}_{i-1})=Q(\mathbf{y}_{k-i})$ for every $i\in [k/2]$,
then $\mathbf h=0$.
\end{lemma}

\begin{proof}
Decompose
\[
\mathbf y_0=\left(\mathbf u^{(1)},\mathbf u^{(3)},\ldots,\mathbf u^{(k-1)}\right)\;,\qquad\mathbf h=\left(\mathbf v^{(1)},\mathbf v^{(3)},\ldots,\mathbf v^{(k-1)}\right)\;,
\]
where $\mathbf u^{(i)},\mathbf v^{(i)}\in\F_p^i$ for every odd
$i\in\{1,3,\ldots,k-1\}$. 

Let $a:=\frac{k-1}{2}\in \F_p$, where division by $2$ is taken in $\F_p$, and define
\[
g(t):=Q(\mathbf y_0+t\mathbf h)=\sum_{1\le i\le k-1,\;i\text{ odd}}P_i\bigl(\mathbf u^{(i)}+t\mathbf v^{(i)}\bigr)\in \F_p[t]\;.
\]
Recall that $Q(\mathbf{y}_{i-1})=Q(\mathbf{y}_{k-i})$ for every $i\in [k/2]$, so $g(t+a)-g(-t+a)$ has at least $k$ roots in $\F_p$: $\{-a,-a+1,\ldots,k-1-a\}$, which are distinct because $p>k$.
Note that $\deg g$ is at most $k-1$, and by the standard fact that a nonzero polynomial of degree at most $k-1$ over a field has at most $k-1$ roots, $g(t+a)-g(-t+a)$ must be the zero polynomial, \textit{i.e.}, $g(t+a)$ is an even polynomial in $t$. 

We now prove by descending induction over the odd integers $r=k-1,k-3,\ldots,1$ that $\mathbf v^{(r)}=0$. Suppose that $\mathbf v^{(s)}=0$ for every odd $s>r$. The corresponding terms
\[
P_s\bigl(\mathbf u^{(s)}+(t+a)\mathbf v^{(s)}\bigr)=P_s\bigl(\mathbf u^{(s)}\bigr)
\]
are then constant in $t$, while every term indexed by $s<r$ has
degree less than $r$. Then by the homogeneity of $P_r$, the coefficient of $t^r$ in $g(t+a)$ is exactly $P_r\bigl(\mathbf v^{(r)}\bigr)$.
Since $g(t+a)$ is an even polynomial in $t$ and $r$ is odd, this coefficient must vanish. Therefore $P_r\bigl(\mathbf v^{(r)}\bigr)=0$. As $P_r$ has no nontrivial zero, it follows that $\mathbf v^{(r)}=0$.
The induction therefore gives
\[
\mathbf h
=
\bigl(\mathbf v^{(1)},\mathbf v^{(3)},\ldots,
\mathbf v^{(k-1)}\bigr)
=0\;.
\]
\end{proof}

\begin{proof}[Proof of Theorem~\ref{thm:cyclic}]
For $n\in\Z/p^{k^2/4}\Z$, write its standard base-$p$ expansion as
\[
    n=x_0(n)+px_1(n)+\cdots+p^{k^2/4-1}x_{k^2/4-1}(n)\;.
\]
Let $\mathbf x(n):=(x_0(n),x_1(n), \ldots, x_{k^2/4-1}(n))$.
Define
\[
    c_k(n):=
    \bigl(
       \tau(x_0(n)),\ldots,\tau(x_{k^2/4-2}(n)),
       Q(\mathbf x(n))
    \bigr)\;.
\]
This coloring uses at most $(k-1)^{k^2/4-1}p$ colors. Suppose that
\[
    n_i=a+id\;,\qquad i=0,1,\ldots,k-1\;,
\]
is symmetrically colored by $c_k$. Applying Lemma~\ref{lem:carry} with $L=k^2/4$, we obtain
\[
    \mathbf x(n_i)=\mathbf x(n_0)+i\mathbf h
\]
for some $\mathbf h\in\F_p^{k^2/4}$. Lemma~\ref{lem:layered-norm} then implies that $\mathbf h=0$. Lemma~\ref{lem:carry} therefore gives $d=0$ in $\Z/p^{k^2/4}\Z$. Hence there is no nontrivial symmetrically colored $k$-AP.
\end{proof}

\begin{remark}
    This proof is a standard application of the polynomial method in additive combinatorics; see, for example, \cite{guth2016polynomial}. In view of the Chevalley–Warning theorem, one should not expect to gain more than $O_k(1)$ dimensions over $k^2/4$ by using another polynomial of degree $O_k(1)$ within the same framework.
\end{remark}

\begin{corollary}\label{cor:interval}
For every fixed even integer $k\ge4$ and every positive integer $N$, the interval $[N]$ admits a coloring with $O_k(N^{4/k^2})$ colors and no nontrivial symmetrically colored $k$-AP. More precisely, for all sufficiently large $N$, at most $2(k-1)^{k^2/4-1}N^{4/k^2}$ colors suffice.
\end{corollary}

\begin{proof}
By the Bertrand--Chebyshev theorem, for every integer $N\ge k^{k^2/4}$, there is a prime $p>k$ with
\[
\lfloor N^{4/k^2}\rfloor< p<2\lfloor N^{4/k^2}\rfloor\;.
\]
Restricting the coloring from Theorem~\ref{thm:cyclic} to $[N]\subseteq \Z/p^{k^2/4}\Z$ gives the result.
\end{proof}

\section{Application to \erdos Problem~160}

In \erdos Problem~160, one must exclude not only the symmetric $ABBA$ pattern, but all color patterns on a nontrivial 4-AP that use at most two colors. Hunter~\cite{Hunter160} observed that the remaining patterns can be eliminated by taking the product with Behrend-style colorings~\cite{Behrend} using $N^{o(1)}$ colors. We give a proof of this observation.

\begin{lemma}[Lemma~3.2 of \cite{DTZ}]\label{lem:behrend-coloring-1}
    The patterns $AAAA$, $AAAB$, $AABA$, $ABAA$, and $BAAA$, on nontrivial $4$-APs in $[N]$, can be eliminated using an $N^{o(1)}$ coloring.
\end{lemma}

\begin{lemma}[Lemma~7.11 of \cite{DTZ}]\label{lem:behrend-coloring-2}
    The pattern $ABAB$ on nontrivial $4$-APs in $[N]$ can be eliminated using an $N^{o(1)}$ coloring.
\end{lemma}

\begin{lemma}[Behrend-style coloring]\label{lem:behrend-coloring-3}
    The pattern $AABB$ on nontrivial $4$-APs in $[N]$ can be eliminated using an $N^{o(1)}$ coloring.
\end{lemma}

\begin{proof}
Let $M=2^{\lceil\sqrt{\log_2 N}\rceil}$ and $m=\lceil\log_M(N+1)\rceil$. For $n\in [N]$, write its standard base-$M$ expansion as $n=\sum_{j=0}^{m-1}x_j(n)M^j$.
Define $\tau(x):=\lfloor\log_2(M-x)\rfloor$, $\Psi(n):=\sum_{j=0}^{m-1}x_j(n)^2$, and color $n$ by $\psi_N(n):=\bigl(\tau(x_0(n)),\ldots,\tau(x_{m-1}(n)),\Psi(n)\bigr)$. The total number is $O((\log M)^mmM^2)=N^{o(1)}$.

We first prove that the carries are consistent. Suppose $y_0,y_1,y_2,y_3$ form a $4$-AP in $\Z/M\Z$ and $\tau(y_0)=\tau(y_1)$, $\tau(y_2)=\tau(y_3)$. If the step size $s> M/2$, then we can reverse the order as in the proof of Lemma~\ref{lem:one-digit}. So we only need to focus on $s\le M/2$. A carry in the first or third addition would place the corresponding pair on opposite sides of $M/2$, so $\tau(y_0)\ne\tau(y_1)$ or $\tau(y_2)\ne\tau(y_3)$. If only the middle addition carried, then $s=y_1-y_0\ge M-y_1$, so $M-y_0\ge 2(M-y_1)$, and thus $\tau(y_0)\ge \tau(y_1)+1$, leading to contradiction.

Now let $n_0,n_1,n_2,n_3$ be a $4$-AP with color pattern $AABB$. The carry consistency established above yields a digitwise $4$-AP $\mathbf x(n_i)=\mathbf u+i\mathbf v$, where $i\in\{0,1,2,3\}$, for some $\mathbf u,\mathbf v\in\mathbb Z^m$. Let $f(t):=\Vert\mathbf u\Vert^2+2t\langle \mathbf u,\mathbf v\rangle+t^2\Vert\mathbf v\Vert^2$, which is a quadratic function. However, the AABB pattern indicates that $f$ has two axes of symmetry, $1/2$ and $5/2$, therefore $f$ must be a constant function, and thus $\mathbf v=0$. Consequently, every $4$-AP with color pattern $AABB$ is trivial.
\end{proof}

\begin{proof}[Proof of Corollary~\ref{cor:erdos}]
Taking the product of the coloring from Corollary~\ref{cor:interval} with $k=4$ and the colorings in Lemmas~\ref{lem:behrend-coloring-1}, \ref{lem:behrend-coloring-2}, and \ref{lem:behrend-coloring-3} uses $N^{1/4+o(1)}$ colors.
\end{proof}

\section{Consequences for Fourier-uniform sets}
Deng, Tidor, and Zhao~\cite{DTZ} proved that an $r$-coloring of $\Z/M\Z$ without symmetrically colored $4$-APs implies
\begin{equation}\label{eq:dtz-reduction}
    \rho_4(\alpha)=O_{r,M}(
    \alpha^{3+\frac12\log_r M})\;,
\end{equation}
for $0<\alpha<1/2$; see~\cite[Theorem~1.8]{DTZ}. They also proved that an $r$-coloring of $\Z/M\Z$ without symmetrically colored $k$-APs implies
\begin{equation}\label{eq:dtz-reduction2}
    \rho_k(\alpha)
    =O_{r,M,k,\varepsilon}\!\left(
      \alpha^{\,k-1+(\log_r M)/(k-1)-\varepsilon}
    \right)\;,
\end{equation}
for even $k\ge 4$ and $0<\alpha<1$; see \cite[Proposition~2.4 and Theorem~2.8]{DTZ}.

\begin{proof}[Proof of Corollary~\ref{cor:rho}]
By Theorem~\ref{thm:cyclic}, the exponent in~\eqref{eq:dtz-reduction} can be taken to be
\[
    3+\frac12\log_{27p}(p^4)
    =3+\frac{2\log p}{\log p+\log27}\;.
\]
This tends to $5$ as $p\to\infty$.
\end{proof}

\begin{proof}[Proof of Corollary~\ref{cor:gowers}]
Applying~\eqref{eq:dtz-reduction2} with $M=p^{k^2/4}$ and $r=(k-1)^{k^2/4-1}p$ gives that for every even $k\geq4$ and every $\varepsilon>0$, the exponent can be taken to be
\[
k-1+\frac{\frac{k^2}{4}\log p}{(k-1)\left(\left(\frac{k^2}{4}-1\right)\log(k-1)+\log p\right)}-\varepsilon\;.
\]
As $p\to\infty$, this tends to
\[
k+\frac{(k-2)^2}{4(k-1)}-\varepsilon\;.
\]
Choosing $\varepsilon<\frac{(k-2)^2}{4(k-1)}$ makes the exponent larger than the random exponent $k$. This provides the first counterexamples to Gowers' conjectured lower bound \cite[Conjecture~4.2]{Gowers} for every even $k\ge6$.
\end{proof}

\section*{Acknowledgements}
The key construction idea for $4$-APs was suggested during an interaction with OpenAI's GPT-5.6 Sol. The authors take full responsibility for all statements and proofs. The authors thank Mingyang Deng, Jonathan Tidor and Yufei Zhao for developing the foundational reduction framework on which this work builds.
RS thanks Zach Hunter for pointing out a related paper that uses a very similar iterated norm coloring. RS thanks Baitian Li for noting, via the Chevalley–Warning theorem, that constant-degree polynomial methods can yield at most a constant-order improvement. RS also thanks Timothy Gowers for kindly sharing his thoughts via email.

\bibliographystyle{alpha}
\bibliography{ref}

\end{document}